\documentclass{iasr}
\newcommand{\lra}{\longrightarrow}

\renewcommand\thisnumber{0}
\renewcommand\thisyear {2010}
\renewcommand\thismonth{February}
\renewcommand\thisvolume{00}
\usepackage[dvips]{graphicx}
\usepackage{amsthm}

\begin{document}

\title{Polynilpotent Capability of Finitely Generated Abelian
Groups}
\author[Mashayekhy B et.~al.]{Mashayekhy Behrooz\affil{1}\comma\corrauth,
       Parvizi Mohsen\affil{1}, and Kayvanfar Saeed\affil{1}\comma\affil{2}}
 \address{\affilnum{1}\ Department of Pure Mathematics, Center of Excellence in Analysis on Algebraic Structures, Ferdowsi University of Mashhad, Mashhad, Iran. \\
           \affilnum{2}\ Institute for Studies in Theoretical Physics and Mathematics, Tehran, Iran.}
 \corraddr{Mashayekhy Behrooz, Department of Pure Mathematics, Center of Excellence in Analysis on Algebraic Structures, Ferdowsi University of Mashhad, P.O.Box 1159-91775, Mashhad, Iran.
          Email: \tt mashf@math.um.ac.ir}
\received{3 october~ 2009}

\newtheorem{thm}{Theorem}[section]
 \newtheorem{cor}[thm]{Corollary}
 \newtheorem{lem}[thm]{Lemma}
 \newtheorem{prop}[thm]{Proposition}
 \newtheorem{defn}[thm]{Definition}
 \newtheorem{rem}[thm]{Remark}
\newtheorem{Example}{Example}[section]

\begin{abstract}
In this paper we determine all finitely generated
abelian groups which are varietal capable with respect to the
variety of polynilpotent groups. This result is a vast
generalization of the famous Baer's result about capability of
finitely generated abelian groups.
\end{abstract}

\maketitle

\section{Introduction}
\label{sec1}
In 1938 Baer \cite{Baer} investigated the conditions for a group to be
the inner automorphism group of another group. Hall and Senior \cite{HallM} called
such groups \textit{capable}. P. Hall in his landmark paper \cite{HallP},
noted that characterization of capable groups is important in classifying
groups of prime power order. Baer \cite{Baer} succeeded to determine all
capable groups among the direct sums of cyclic groups, and hence
finitely generated abelian groups. The third author in his joint paper in 1997 \cite{Moghaddam}, generalized the notion
of capability to any variety of groups which was called $\mathcal{V}$-capability.

The aim of this work is determining
all finitely generated abelian groups which are varietal capable
with respect to the variety of polynilpotent groups of class row
$(c_1,c_2,\ldots,c_t)$, that is $\mathcal{N}_{c_1,c_2,\ldots,c_t}$.
This result is a vast generalization of
Baer's work in the case of finitely generated abelian groups.

\section{Preliminaries}
\label{sec2}
We shall assume that the reader is
familiar with the notion of variety of groups and the verbal and the
marginal subgroups. The notion of Baer invariant is also
considered to be known (see \cite{Moghaddam}).
The following definition is of several usage in this article.
\begin{defn}
Let $\mathcal{V}$ be any variety of groups defined by a set of laws $V$, and $G$ be any
group. $G$ is called $\mathcal{V}$-capable if there exists a group $E$
which satisfies $G\cong E/V^*(E)$,
where $V^*(E)$ is the marginal subgroup of $E$ with respect to $\mathcal{V}$.
\end{defn}
It should be noted that although two different sets of words can generate a variety but the marginal subgroups with respect to these sets of words coincide, so the terminology $\mathcal{V}$-capability has no ambiguity.
According to the above definition capable groups are $\mathcal{A}$-capable , where $\mathcal{A}$ is the variety of abelian groups.

The following definition and theorem are found in \cite{Moghaddam} and contains a
necessary and sufficient condition for a group to be
$\mathcal{V}$-capable.
\begin{defn}
Let $\mathcal{V}$
be any variety and $G$ be any group. Then $V^{**}(G)$ is defined as follows.
$$V^{**}(G)=\cap \{\psi(V^*(E))\ | \ \psi:E \stackrel{onto}\longrightarrow G \ , \ ker\psi\subseteq V^*(E)\}.$$
\end{defn}
\begin{thm}
With the above notations and assumptions $G/V^{**}(G)$ is the largest quotient of $G$ which is ${\cal
V}$-capable, and hence $G$ is $\mathcal{V}$-capable if and only if
$V^{**}(G)=1$.
\end{thm}
The following theorem and its consequences state relationship
between $\mathcal{V}$-capability and Baer invariants.
\begin{thm} (\cite{Moghaddam}).
Let $\mathcal{V}$ be any variety
of groups, $G$ be any group, and $N$ be a normal subgroup of $G$
contained in the marginal subgroup $V^*(G)$. Then the natural homomorphism
$\mathcal{V}M(G)\longrightarrow \mathcal{V}M(G/N)$ is injective if and
only if $N\subseteq V^{**}(G)$, where $\mathcal{V}M(G)$ is the Baer invariant of $G$ with respect to $\mathcal{V}$.
\end{thm}
In finite case the following corollary is useful to be applied.
\begin{cor} (\cite{Parvizi}).
Let $\mathcal{V}$ be any variety
and $G$ be any group with $V(G)=1$. Assume that $\mathcal{V}M(G)$ and
$\mathcal{V}M(G/N)$ are finite groups for a normal subgroup $N$ of $G$,
then the natural homomorphism $\mathcal{V}M(G)\longrightarrow \mathcal{V}M(G/N)$ is
injective if and only if $|\mathcal{V}M(G/N)|=|\mathcal{V}M(G)|$.
\end{cor}
The following result is a consequence of Theorem 2.4.
\begin{cor}
An abelian group $G$ is
${\cal V}$-capable if and only if the natural homomorphism ${\cal
V}M(G)\longrightarrow {\cal V}M(G/<x>)$ has a
non-trivial kernel for all non-identity elements $x$ in $G$.
\end{cor}
In order to use Theorem 2.4 we need the explicit
structure of the Baer invariants of finitely generated abelian
groups with respect to
the variety of polynilpotent groups. The following give us the structure.
\begin{thm} (\cite{Mashayekhy}).
Let ${\cal
N}_{c_1,c_2,\ldots,c_t}$ be the polynilpotent variety of class row
$(c_1,c_2,\ldots,c_t)$ and $G\cong \mathbb{Z}^{(m)}\oplus \mathbb{Z}_{n_1}\oplus\ldots\oplus \mathbb{Z}_{n_k}$ be a finitely generated
abelian group, where $n_{i+1}\mid n_i$ for all $1\leq i \leq k-1$.
Then an explicit structure of the polynilpotent multiplier of $G$ is
as follows.
$$\mathcal{N}_{c_1,c_2,\ldots,c_t}M(G)\cong \mathbb{Z}^{(f_m)}\oplus \mathbb{Z}_{n_1}^{(f_{m+1}-f_m)}\oplus\ldots\oplus
\mathbb{Z}_{n_k}^{(f_{m+k}-f_{m+k-1})},$$ where
$f_i=\chi_{c_t+1}(\chi_{c_{t-1}+1}(\ldots(\chi_{c_1+1}(i))\ldots))$
for all $m\leq i \leq m+k$.
\end{thm}
To use Theorem 2.4 and Corollary
2.5, we need the structure of subgroups of finitely generated
abelian groups as follows.
\begin{thm} (\cite{Parvizi}).
Let  $G\cong \mathbb{Z}^{(m)}\oplus \mathbb{Z}_{n_1}\oplus\ldots\oplus \mathbb{Z}_{n_k}$ be a
finitely generated abelian group, where $n_{i+1}\mid n_i$ for all
$1\leq i \leq k-1$ and $H\leq G$ be a finite subgroup of $G$. Then
$H\cong \mathbb{Z}_{m_1}\oplus\ldots\oplus \mathbb{Z}_{m_k}$, where
$m_{i+1}\mid m_i$ for all $1\leq i \leq k-1$
and $m_i\mid n_i$ for all $1\leq i \leq k$.
\end{thm}

\section{Main Results}
First of all we state two theorems of Baer, Burns and Ellis, which
determine all $\mathcal{N}_c$-capable finitely generated abelian
groups, in which $\mathcal{N}_c$ is the variety of nilpotent groups
of class at most $c$. The case $c=1$ is permissible and refers to
capability.\\
The following theorem is not just what Baer proved, but a
conclusion of it as a special case.
\begin{thm}
(R. Baer \cite{Baer}). Let $G\cong \mathbb{Z}^{(m)}\oplus \mathbb{Z}_{n_1}\oplus\ldots\oplus \mathbb{Z}_{n_k}$ be a
finitely generated abelian group, where $n_{i+1}\mid n_i$ for all
$1\leq i \leq k-1$. Then $G$ is capable
precisely when \\
 $(i)$ $m\geq 2$; or\\
$(ii)$ $m=0$, $k\geq 2$ and $n_1=n_2$.
\end{thm}
\begin{thm}
(J. Burns and G. Ellis \cite{Burns}). A
finitely generated abelian group is ${\cal N}_c$-capable if and
only if it is capable.
\end{thm}
According to the above theorem, the condition for a finitely
generated abelian group to be $\mathcal{N}_c$-capable is just the
condition for capability of it. It is of interest to know which
varieties behave like as nilpotent varieties in case of capability
of finitely generated abelian groups. Answering this question in
general case need much more information, but in case of some outer
commutator varieties, more precisely polynilpotent varieties, we intend to
answer it. In the rest we see that not all polynilpotent varieties
behave like as nilpotent varieties in case of capability, but a
large class of them do. To show
varieties which not behave so we give the following example.
\begin{example}
Let $G\cong \mathbb{Z}_n\oplus
\mathbb{Z}_n$, where $n\in \mathbb{N}$. One may easily show that $G$ and
all its quotients are at most two-generator abelian groups, so by
Theorem 2.7 $\mathcal{S}_2M(G/N)=0$ for any normal subgroup of $G$, $N$
say, in which $\mathcal{S}_2$ is the variety of metabelian groups. Now
Theorem 2.4 implies $G$ is not $\mathcal{S}_2$-capable.
\end{example}
Now the main goal of this article is presented, which is
classifying all $\mathcal{N}_{c_1,c_2,\ldots,c_t}$-capable groups in
the class of finitely generated abelian groups.
According to Theorem 2.7 if $t\geq 2$ and $c_1=1$, then $\mathcal{N}_{c_1,c_2,\ldots,c_t}M(G)=0$ whenever $G$ has at most two
generators, but if \ $t=1$\  or\  $c_1\geq 2$\ , then  $\mathcal{N}_{c_1,c_2,\ldots,c_t}M(G)=0$ only when $G$ is cyclic. Hence we
need to separate these two cases, and the rest of work shows
the results are different.

We only state the proof in the case $t\geq 2$ and $c_1=1$, the proof of
the other case is similar. Just for the sake of clarity
we separate the procedure in two cases; finite and infinite.

Case one: Finite case.\\
Throughout the following two lemmas we assume that $t\geq 2$ and $c_1=1$.
\begin{lem}
Let $G\cong \mathbb{Z}_{n_1}\oplus\ldots\oplus \mathbb{Z}_{n_k}$ be a finite abelian group,
where $n_{i+1}\mid n_i$ for all $1\leq i \leq k-1$. If $G$ is $\mathcal{N}_{c_1,c_2,\ldots,c_t}$-capable then
$k\geq 3$ and $n_1=n_2=n_3$.
\end{lem}
\begin{proof}
By contrary assume we have $k\leq 2$
or not $n_1=n_2=n_3$. If $k\leq 2$ then similar to Example 3.1 $G$
and all its quotients are at most two-generator abelian groups so by
Theorem 2.7 $\mathcal{N}_{c_1,c_2,\ldots,c_t}M(G/N)=0$ for any normal
subgroup $N$ of $G$, so by Theorem 2.4 $G$ is not $\mathcal{N}_{c_1,c_2,\ldots,c_t}$-capable. On the other hand if, for example,
$n_1\neq n_2$, we introduce a normal subgroup of $G$, $N$ say; for
which the natural homomorphism $\mathcal{N}_{c_1,c_2,\ldots,c_t}M(G)\longrightarrow \mathcal{N}_{c_1,c_2,\ldots,c_t}M(G/N)$ is injective. Put $N$ the subgroup generated by $({\bar n_2},{\bar 0},\ldots,{\bar 0})$. It is easy to
see that $N\cong \mathbb{Z}_{\frac {n_1}{n_2}}$ and $G/N\cong \mathbb{Z}_{n_2}\oplus \mathbb{Z}_{n_2}\ldots\oplus \mathbb{Z}_{n_k}$. Now Theorem 2.7 shows that $|\mathcal{N}_{c_1,c_2,\ldots,c_t}M(G)|=|\mathcal{N}_{c_1,c_2,\ldots,c_t}M(G/N)|$, hence by Corollary 2.5  we have the
injectivity of $\mathcal{N}_{c_1,c_2,\ldots,c_t}M(G)\longrightarrow
\mathcal{N}_{c_1,c_2,\ldots,c_t}M(G/N)$. So the result holds by Theorem
2.4. The second case for which we have $n_1=n_2\neq n_3$ has a
similar proof by putting $N=<({\bar 0},{\bar n_3},{\bar
0},\ldots,{\bar 0})>$.
\end{proof}
The above lemma presents a necessary condition for a finite
abelian group to be $\mathcal{N}_{c_1,c_2,\ldots,c_t}$-capable. The
following shows it is sufficient.
\begin{lem}
Let $G\cong \mathbb{Z}_{n_1}\oplus\ldots\oplus \mathbb{Z}_{n_k}$ be a finite abelian group,
where $n_{i+1}\mid n_i$ for all $1\leq i \leq k-1$ in which $k\geq
3$ and $n_1=n_2=n_3$. Then $G$ is $\mathcal{N}_{c_1,c_2,\ldots,c_t}$-capable.
\end{lem}
\begin{proof}
By Theorem 2.4 it is enough to show
that for an arbitrary normal subgroup $N$ of $G$, $\mathcal{N}_{c_1,c_2,\ldots,c_t}M(G)\lra \mathcal{N}_{c_1,c_2,\ldots,c_t}M(G/N)$
is not injective. In finite abelian groups each quotient is
isomorphic to a subgroup and vice versa. Now let $N$ be an arbitrary
normal subgroup of $G$, then $G/N$ is isomorphic to a subgroup of
$G$, $H$ say; so by Theorem 2.8 $H\cong \mathbb{Z}_{m_1}\oplus\ldots\oplus \mathbb{Z}_{m_k}$, where $m_{i+1}\mid m_i$
for all $1\leq i \leq k-1$ and $m_i\mid n_i$ for all $1\leq i \leq
k$. Computing $\mathcal{N}_{c_1,c_2,\ldots,c_t}M(G)$ and $\mathcal{N}_{c_1,c_2,\ldots,c_t}M(H)$ using Theorem 2.7, shows that $|\mathcal{N}_{c_1,c_2,\ldots,c_t}M(G)|=|\mathcal{N}_{c_1,c_2,\ldots,c_t}M(H)|$ if
and only if $m_i=n_i$ for all $3\leq i \leq k$, but $n_1=n_2=n_3$ by
hypothesis which implies $n_1=m_1$ and $n_2=m_2$ which is equivalent
to $H=G$ or $N=0$. Now by Theorem 2.4 $G$ is $\mathcal{N}_{c_1,c_2,\ldots,c_t}$-capable.
\end{proof}
Case two: Infinite case.\\
Similar to the previous case, through the following two lemmas we
assume that $t\geq 2$ and $c_1=1$.
\begin{lem}
Let $G\cong \mathbb{Z}^{(m)}\oplus
\mathbb{Z}_{n_1}\oplus\ldots\oplus \mathbb{Z}_{n_k}$ be an infinite
finitely generated abelian group, where $n_{i+1}\mid n_i$ for all
$1\leq i \leq k-1$. If $G$ is
$\mathcal{N}_{c_1,c_2,\ldots,c_t}$-capable then $m\geq 3$.
\end{lem}
\begin{proof}
We show that if $m\leq 2$, then there
exists a normal subgroup $N$ of $G$ for which the natural
homomorphism $\mathcal{N}_{c_1,c_2,\ldots,c_t}M(G)\longrightarrow \mathcal{N}_{c_1,c_2,\ldots,c_t}M(G/N)$ is injective and hence the result
follows by Theorem 2.4.
Suppose $m=2$, then $G\cong \mathbb{Z}\oplus \mathbb{Z}\oplus \mathbb{Z}_{n_1}\oplus\ldots\oplus \mathbb{Z}_{n_k}$. Put $N$ to be the
subgroup generated by $(0,n_1,{\bar 0},\ldots,{\bar 0})$ so $N\cong n_1\mathbb{Z}$ and $G/N\cong \mathbb{Z}\oplus \mathbb{Z}_{n_1}\oplus \mathbb{Z}_{n_1}\oplus\ldots\oplus \mathbb{Z}_{n_k}$. Now by Theorem 2.7 we
have $|\mathcal{N}_{c_1,c_2,\ldots,c_t}M(G)|=|\mathcal{N}_{c_1,c_2,\ldots,c_t}M(G/N)|$. Using Theorems 2.4 and Corollary 2.5 the
result will follow. If $m=1$ then put $N=<(n_1,{\bar
0},\ldots,{\bar 0})>$ and use a similar proof.
\end{proof}
The reverse of the above lemma can be presented as follows.
\begin{lem}
Let $G\cong \mathbb{Z}^{(m)}\oplus
\mathbb{Z}_{n_1}\oplus\ldots\oplus \mathbb{Z}_{n_k}$ be an infinite
finitely generated abelian group, where $n_{i+1}\mid n_i$ for all
$1\leq i \leq k-1$. If $m\geq 3$
then $G$ is $\mathcal{N}_{c_1,c_2,\ldots,c_t}$-capable.
\end{lem}
\begin{proof}
Similar to Lemma 3.4 it is enough to
show that there is no normal subgroup $N$ of $G$ for which $\mathcal{N}_{c_1,c_2,\ldots,c_t}M(G)\lra \mathcal{N}_{c_1,c_2,\ldots,c_t}M(G/N)$
is injective. Consider two cases
$(i)$  $N$ is infinite; and
$(ii)$ $N$ is finite.

In case $(i)$ trivially  $r_0(G/N)<r_0(G)$, now by Theorem 2.7 $r_0(\mathcal{N}_{c_1,c_2,\ldots,c_t}M(G/N))<r_0(\mathcal{N}_{c_1,c_2,\ldots,c_t}M(G))$. So there is no injection between $\mathcal{N}_{c_1,c_2,\ldots,c_t}M(G)$ and\linebreak $\mathcal{N}_{c_1,c_2,\ldots,c_t}M(G/N)$.
In case $(ii)$ $N\leq t(G)$ and hence $G/N\cong \mathbb{Z}^{(m)}\oplus
\mathbb{Z}_{m_1}\oplus\ldots\oplus \mathbb{Z}_{m_k}$ where $m_{i+1}\mid
m_i$ and $m_i\mid n_i$ for all $1\leq i \leq k-1$, now by Theorem
2.7 we have $\mathcal{N}_{c_1,c_2,\ldots,c_t}M(G)\cong \mathbb{Z}^{(f_m)}\oplus \mathbb{Z}_{n_1}^{(f_{m+1}-f_m)}\oplus\ldots\oplus
\mathbb{Z}_{n_k}^{(f_{m+k}-f_{m+k-1})},$ and $\mathcal{N}_{c_1,c_2,\ldots,c_t}M(G/N)\cong \mathbb{Z}^{(f_m)}\oplus \mathbb{Z}_{m_1}^{(f_{m+1}-f_m)}\oplus\ldots\oplus \mathbb{Z}_{m_k}^{(f_{m+k}-f_{m+k-1})}$. Now by contrary assume that the natural homomorphism $\mathcal{N}_{c_1,c_2,\ldots,c_t}M(G)\longrightarrow \mathcal{N}_{c_1,c_2,\ldots,c_t}M(G/N)$ which will be called $f$ through
the proof, is injective. It is easy to show that $t(\mathcal{N}_{c_1,c_2,\ldots,c_t}M(G))=\mathbb{Z}_{n_1}^{(f_{m+1}-f_m)}\oplus\ldots\oplus \mathbb{Z}_{n_k}^{(f_{m+k}-f_{m+k-1})}$ and $t(\mathcal{N}_{c_1,c_2,\ldots,c_t}M(G/N))=\mathbb{Z}_{m_1}^{(f_{m+1}-f_m)}\oplus\ldots\oplus \mathbb{Z}_{m_k}^{(f_{m+k}-f_{m+k-1})}$. Also we have $f(t(\mathcal{N}_{c_1,c_2,\ldots,c_t}M(G)))\subseteq t(\mathcal{N}_{c_1,c_2,\ldots,c_t}M(G/N))$. As $f$ is injective we must have
$m_i=n_i$ for all $1\leq i \leq k$ so $t(G)=t(G/N)$ but
$t(G/N)=t(G)/N$ and hence $N=0$ which completes the proof.
\end{proof}
The latest lemmas classify all finitely generated abelian groups
which are $\mathcal{N}_{c_1,c_2,\ldots,c_t}$-capable, provided that
$t\geq 2$ and $c_1=1$. Of the most famous examples of these
varieties are $\mathcal{S}_{\ell}$'s, the varieties of
soluble groups of length at most ${\ell}$.

The case $c_1\geq 2$ or $t=1$ has different results but similar
proof, so the proof are shortened. Examples of these varieties are $\mathcal{N}_c$'s and $\mathcal{A}$,
which are the varieties of nilpotent groups of class at most $c$
and abelian groups, respectively, for the case $t=1$, and polynilpotent varieties
which are not solvable for the case $c_1\geq 2$.
The following theorems are stated with the conditions $c_1\geq 2$
or $t=1$.
\begin{thm}
Let $G\cong \mathbb{Z}_{n_1}\oplus\ldots\oplus \mathbb{Z}_{n_k}$ be a finite abelian group,
where $n_{i+1}\mid n_i$ for all $1\leq i \leq k-1$. Then $G$ is
$\mathcal{N}_{c_1,c_2,\ldots,c_t}$-capable if and
only if $k\geq 2$ and $n_1=n_2$.
\end{thm}
\begin{proof}
The proof is completely similar to the
proofs of Lemmas 3.3 and 3.4. The only matter which may be
considered is the fact that for an arbitrary finite abelian group
$G\cong \mathbb{Z}_{n_1}\oplus\ldots\oplus \mathbb{Z}_{n_k}$ we have
$\mathcal{N}_{c_1,c_2,\ldots,c_t}M(G)\cong \mathbb{Z}_{n_2}^{(f_2)}\oplus
\mathbb{Z}_{n_3}^{(f_3-f_2)} \ldots\oplus \mathbb{Z}_{n_k}^{(f_k-f{k-1})}$.
\end{proof}
\begin{thm}
Let $G\cong \mathbb{Z}^{(m)}\oplus
\mathbb{Z}_{n_1}\oplus\ldots\oplus \mathbb{Z}_{n_k}$ be an infinite
finitely generated abelian group, where $n_{i+1}\mid n_i$ for all
$1\leq i \leq k-1$ is ${\cal
N}_{c_1,c_2,\ldots,c_t}$-capable if and only if $m\geq 2$.
\end{thm}
\begin{proof}
The proof is similar to Lemmas 3.5 and 3.6. The only fact may be considered is the structure of the
polynilpotent multipliers of $G$ and $G/N$ in the case $c_1\geq 2$
or $t=1$. More exactly the difference is the fact that we have
$f_2\neq 0$, so the torsion free rank of $\mathcal{N}_{c_1,c_2,\ldots,c_t}M(G)$ is zero exactly when $G$ has the
torsion free rank 1.
\end{proof}

\section*{Acknowledgments}
The authors are grateful to the referees for their careful works and suggestions that improved this paper.\\
The third author was in part supported by a grant from IPM (No. 83200037).


\end{document}